\documentclass[12pt]{amsart}
\usepackage{graphicx} 
\usepackage[margin=1in]{geometry}
\usepackage{amsmath, amsfonts, amsthm, amssymb}
\usepackage{hyperref}  
\usepackage{color} 
\usepackage{array} 
\usepackage{float} 
\usepackage[shortlabels]{enumitem}
\usepackage{ytableau}
\ytableausetup{centertableaux} 

\usepackage{mathtools}
\usepackage{cancel}
\usepackage{tikz} 
\usetikzlibrary{trees} 
\usepackage{pgfplots}
\usetikzlibrary{fillbetween}

\usepackage[english]{babel}
\usepackage[autostyle]{csquotes}

\newtheorem{theorem}{Theorem}[section]
\newtheorem{lemma}[theorem]{Lemma}
\newtheorem{proposition}[theorem]{Proposition}
\newtheorem{corollary}[theorem]{Corollary}

\newtheorem{problem}[theorem]{Problem}

\theoremstyle{definition}

\newtheorem{example}[theorem]{Example}

\title{Symmetry in Equivariant cohomology of $\mathbb{P}^n$}

\author{Duy Phan}
\address{Dept. of Mathematics, U. Illinois at Urbana-Champaign, Urbana, IL 61801, USA}
\email{duyphan2@illinois.edu}
\date{\today}

\begin{document}

\maketitle

\begin{abstract}
We resolve a problem of Anderson and Fulton by providing a symmetric and positive product rule for the equivariant cohomology of projective space.
\end{abstract}

\section{Introduction}

Let $X = \mathbb{P}^n$ be complex projective space. 
For each $k = 0,1,\dots,n$, define
\[
X_k := \{ [x_0:\cdots:x_n] \in \mathbb{P}^n \mid x_{0} = \cdots = x_{k-1} = 0 \} \cong \mathbb{P}^{n-k} ,
\]
the standard linear subspace of codimension $k$. 
Let $T \subset \mathrm{GL}_{n+1}$ be the diagonal torus acting naturally on $X$. 
The subvarieties $X_k$ are $T$-stable, and hence determine equivariant cohomology classes (see, e.g., Anderson and Fulton~\cite[Chapter~4.3]{MR4655919})
\[
\sigma_k := [X_k]_T \in H_T^*(\mathbb{P}^n).
\]
These classes form a basis of $H_T^*(\mathbb{P}^n)$ as a module over
\[
H_T^*(\mathrm{pt}) \cong \mathbb{Z}[t_1,\dots,t_{n+1}].
\]
The equivariant structure coefficients $C_{i,j}^k \in H_T^*(\mathrm{pt})$ are defined by
\[
\sigma_i \cdot \sigma_j = \sum_{k=0}^n C^k_{i,j} \, \sigma_k.
\]

By a result of Graham~\cite[Theorem~3.1]{MR1853356}, the structure constants are positive when expressed in the simple root variables $\beta_m := t_m - t_{m+1}$. 
Anderson and Fulton~\cite[p.~60]{MR4655919} asked for an explicit combinatorial formula for the coefficients $C^k_{i,j}$ that is both symmetric (i.e., $C^k_{i,j} = C^k_{j,i}$) and manifestly positive in the variables $\{\beta_m\}$. 
To the best of our knowledge, this paper provides the first solution to this question.\\

We now state our main result. Let $i,j,k$ be nonnegative integers and set $r := i+j-k$. Consider the bipartite graph $\mathcal{G}^i_j$ with vertex sets
\[
A = \{a_1,a_2,\ldots,a_i\}, \qquad 
B = \{b_1,b_2,\ldots,b_j\}.
\]
An \emph{$r$-matching} of $\mathcal{G}^i_j$ is a collection of $r$ pairwise vertex-disjoint edges between $A$ and $B$.

\begin{figure}[h]
\centering
\begin{tikzpicture}[scale=0.5, every node/.style={font=\small}]

\fill (0,0) circle (3pt) node[below] {$b_4$};
\fill (2,0) circle (3pt) node[below] {$b_3$};
\fill (4,0) circle (3pt) node[below] {$b_2$};
\fill (6,0) circle (3pt) node[below] {$b_1$};

\fill (1,3) circle (3pt) node[above] {$a_1$};
\fill (3,3) circle (3pt) node[above] {$a_2$};
\fill (5,3) circle (3pt) node[above] {$a_3$};

\draw[line width=1.5pt] (0,0) -- (5,3);
\draw[line width=1.5pt] (4,0) -- (1,3);

\end{tikzpicture}
\caption{A $2$-matching of the bipartite graph $\mathcal{G}^3_4$}
\label{fig:ex}
\end{figure}
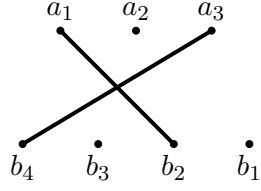

For an edge $e = a_g b_h$ in an $r$-matching $M$, define its \emph{crossing value} by
\begin{equation}
\label{eq-cross}
X(e) \coloneqq \#\{ a_p b_q \in M \mid p < g,\; q < h \}.
\end{equation}
Define the \emph{weight} of an edge $e = a_g b_h$ by
\begin{equation}
\label{eq-wt-edge}
\operatorname{wt}(e) \coloneqq g + h - X(e) - 1,
\end{equation}
and extend this multiplicatively to an $r$-matching $M$ by
\begin{equation}
\label{eq-wt-matching}
\operatorname{wt}(M) \coloneqq \prod_{e \in M} \beta_{\operatorname{wt}(e)}.
\end{equation}

\begin{example}
\label{ex:computation}
In Figure~\ref{fig:ex}, we have $X(a_1 b_2)=0$ and $X(a_3 b_4)=1$. Hence,
\[
\begin{aligned}
\operatorname{wt}(a_1 b_2) &= 1 + 2 - 0 - 1 = 2, \\
\operatorname{wt}(a_3 b_4) &= 3 + 4 - 1 - 1 = 5.
\end{aligned}
\]
Therefore, the weight of the $2$-matching shown in Figure~\ref{fig:ex} is $\operatorname{wt}(M) = \beta_2 \beta_5$.
\end{example}

\begin{theorem}
\label{thm:main}
We have
\[
C^k_{i,j} = \sum_{M} \operatorname{wt}(M),
\]
where the sum runs over all $r$-matchings of $\mathcal{G}^i_j$.
\end{theorem}

This rule is manifestly positive in the simple root variables $\{\beta_m\}$. Moreover, the symmetry $C_{i,j}^k = C_{j,i}^k$ follows from the invariance of the weight of any $r$-matching under a $180^\circ$ rotation of the bipartite graph $\mathcal{G}^i_j$ to $\mathcal{G}^j_i$.

\begin{example}
For $i=2$, $j=3$, and $k=3$, we have $r = i + j - k = 2$. We enumerate all $2$-matchings of the bipartite graph $\mathcal{G}^2_3$ and compute their weights (see Figure~\ref{fig:full-ex}). Summing over all such matchings, we obtain
\[
C_{2,3}^3
=
\beta_3^2
+ 2\beta_3\beta_2
+ \beta_3\beta_1
+ \beta_2^2
+ \beta_2\beta_1.
\]
\end{example}

\begin{figure}[h]
\centering
\begin{tikzpicture}[scale=0.3]

\fill (0,0) circle (5pt) node[below] {\scriptsize $b_3$};
\fill (2,0) circle (5pt) node[below] {\scriptsize $b_2$};
\fill (4,0) circle (5pt) node[below] {\scriptsize $b_1$};
\fill (1,3) circle (5pt) node[above] {\scriptsize $a_1$};
\fill (3,3) circle (5pt) node[above] {\scriptsize $a_2$};

\fill (8,0) circle (5pt) node[below] {\scriptsize $b_3$};
\fill (10,0) circle (5pt) node[below] {\scriptsize $b_2$};
\fill (12,0) circle (5pt) node[below] {\scriptsize $b_1$};
\fill (9,3) circle (5pt) node[above] {\scriptsize $a_1$};
\fill (11,3) circle (5pt) node[above] {\scriptsize $a_2$};

\fill (16,0) circle (5pt) node[below] {\scriptsize $b_3$};
\fill (18,0) circle (5pt) node[below] {\scriptsize $b_2$};
\fill (20,0) circle (5pt) node[below] {\scriptsize $b_1$};
\fill (17,3) circle (5pt) node[above] {\scriptsize $a_1$};
\fill (19,3) circle (5pt) node[above] {\scriptsize $a_2$};

\fill (24,0) circle (5pt) node[below] {\scriptsize $b_3$};
\fill (26,0) circle (5pt) node[below] {\scriptsize $b_2$};
\fill (28,0) circle (5pt) node[below] {\scriptsize $b_1$};
\fill (25,3) circle (5pt) node[above] {\scriptsize $a_1$};
\fill (27,3) circle (5pt) node[above] {\scriptsize $a_2$};

\fill (32,0) circle (5pt) node[below] {\scriptsize $b_3$};
\fill (34,0) circle (5pt) node[below] {\scriptsize $b_2$};
\fill (36,0) circle (5pt) node[below] {\scriptsize $b_1$};
\fill (33,3) circle (5pt) node[above] {\scriptsize $a_1$};
\fill (35,3) circle (5pt) node[above] {\scriptsize $a_2$};

\fill (40,0) circle (5pt) node[below] {\scriptsize $b_3$};
\fill (42,0) circle (5pt) node[below] {\scriptsize $b_2$};
\fill (44,0) circle (5pt) node[below] {\scriptsize $b_1$};
\fill (41,3) circle (5pt) node[above] {\scriptsize $a_1$};
\fill (43,3) circle (5pt) node[above] {\scriptsize $a_2$};

\draw[line width=1pt] (0,0) -- (1,3);
\draw[line width=1pt] (2,0) -- (3,3);

\draw[line width=1pt] (10,0) -- (9,3);
\draw[line width=1pt] (8,0) -- (11,3);

\draw[line width=1pt] (16,0) -- (17,3);
\draw[line width=1pt] (20,0) -- (19,3);

\draw[line width=1pt] (24,0) -- (27,3);
\draw[line width=1pt] (28,0) -- (25,3);

\draw[line width=1pt] (34,0) -- (33,3);
\draw[line width=1pt] (36,0) -- (35,3);

\draw[line width=1pt] (42,0) -- (43,3);
\draw[line width=1pt] (44,0) -- (41,3);

\node at (2,-3) {\small $\beta_3 \beta_3$};
\node at (10,-3) {\small $\beta_2 \beta_3$};
\node at (18,-3) {\small $\beta_3 \beta_2$};
\node at (26,-3) {\small $\beta_1 \beta_3$};
\node at (34,-3) {\small $\beta_2 \beta_2$};
\node at (42,-3) {\small $\beta_1 \beta_2$};

\end{tikzpicture}
\caption{All $2$-matchings of $\mathcal{G}^2_3$ and their weights}
\label{fig:full-ex}
\end{figure}
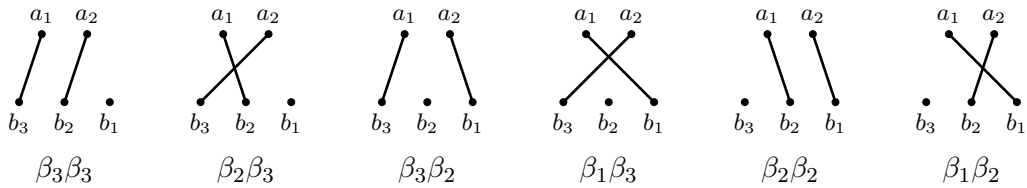

The proof of Theorem~\ref{thm:main} is given in Section~\ref{sec:proof}. In Section~\ref{subsec:app}, we develop several consequences of our rule, including a description of the Newton polytope of $C_{i,j}^k$, the saturated Newton polytope (SNP) property, and a strengthening of the equivariant saturation phenomenon describing the behavior of supports under scaling. Finally, in Sections~\ref{subsec:puzzle} and~\ref{subsec:edge-labeled}, we compare our formulation with existing combinatorial models, namely the Knutson--Tao puzzles and the Thomas--Yong edge-labeled tableaux.

\section{Proof of the bipartite graph rule (Theorem~\ref{thm:main})}
\label{sec:proof}

We begin by recalling two standard results in equivariant Schubert calculus on Grassmannians, which we include for completeness. We then use these to recover the classical explicit formula for the coefficients \(C_{i,j}^k\), and show how our bipartite matching model naturally reproduces each term in this expansion.

\begin{lemma}
\label{lemma-pieri}
We have the equivariant Pieri rule for $\mathbb{P}^n$:
\[
\sigma_1 \cdot \sigma_k = \sigma_{k+1} + (t_1 - t_{k+1})\, \sigma_k.
\]
\end{lemma}

\begin{proof}
By Anderson and Fulton~\cite[p.~59]{MR4655919}, we have
\[
\sigma_k = \prod_{i=1}^k (\zeta + t_i) \in H_T^*(\mathbb{P}^n),
\]
where $\zeta = c_1^T(\mathcal{O}(1))$ is the equivariant hyperplane class, and
$t_i = c_1^T(L_i)$ denotes the weight of the $T$-action on the $i$th coordinate line. Thus,
\[
\sigma_1 \cdot \sigma_k 
= (\zeta + t_1)\prod_{i=1}^k (\zeta + t_i).
\]
Adding and subtracting $t_{k+1}$, we rewrite
\[
\zeta + t_1 = (\zeta + t_{k+1}) + (t_1 - t_{k+1}),
\]
and hence
\[
\begin{aligned}
\sigma_1 \cdot \sigma_k 
&= \bigl[(\zeta + t_{k+1}) + (t_1 - t_{k+1})\bigr]
\prod_{i=1}^k (\zeta + t_i) \\
&= (\zeta + t_{k+1})\prod_{i=1}^k (\zeta + t_i)
+ (t_1 - t_{k+1}) \prod_{i=1}^k (\zeta + t_i) \\
&= \sigma_{k+1} + (t_1 - t_{k+1})\, \sigma_k,
\end{aligned}
\]
as claimed.
\end{proof}

\begin{lemma}
\label{lemma-recurrence}
We have the recursion
\[
C_{i+1,j}^{k} = (t_{i+1}-t_{k+1}) \, C_{i,j}^{k} + C_{i,j}^{k-1}.
\]
\end{lemma}

\begin{proof}
Using the equivariant Pieri rule (Lemma~\ref{lemma-pieri}), we compute
\[
\begin{aligned}
\sigma_1 \cdot (\sigma_i \cdot \sigma_j) 
&= \sigma_1 \cdot \left( \sum_{k} C_{i,j}^{k} \, \sigma_k \right) \\
&= \sum_{k} C_{i,j}^k \, (\sigma_1 \cdot \sigma_k) \\
&= \sum_{k} C_{i,j}^k \, \bigl[ \sigma_{k+1} + (t_1-t_{k+1}) \, \sigma_k \bigr] \\
&= \sum_{k} \left[ C_{i,j}^{k-1} + (t_1 - t_{k+1}) \, C_{i,j}^{k} \right] \, \sigma_k.
\end{aligned}
\]
On the other hand,
\[
\begin{aligned}
(\sigma_1 \cdot \sigma_i ) \cdot \sigma_j 
&= \bigl[ \sigma_{i+1} + (t_1 - t_{i+1}) \, \sigma_i \bigr] \cdot \sigma_j \\
&= \sigma_{i+1}\cdot \sigma_j + (t_1-t_{i+1})(\sigma_i\cdot \sigma_j) \\
&= \sum_k C_{i+1,j}^k\,\sigma_k + (t_1-t_{i+1})\sum_k C_{i,j}^k\,\sigma_k \\
&= \sum_k \left[ C_{i+1,j}^k + (t_1-t_{i+1}) \, C_{i,j}^k \right]\sigma_k.
\end{aligned}
\]
By associativity,
\[
\sigma_1 \cdot (\sigma_i \cdot \sigma_j)= (\sigma_1 \cdot \sigma_i ) \cdot \sigma_j.
\]
Comparing coefficients of $\sigma_k$ gives
\[
C_{i,j}^{k-1} + (t_1-t_{k+1}) \, C_{i,j}^k
=
C_{i+1,j}^k + (t_1-t_{i+1}) \, C_{i,j}^k.
\]
Rearranging yields
\[
C_{i+1,j}^{k}
=
(t_{i+1}-t_{k+1}) \, C_{i,j}^{k}
+ C_{i,j}^{k-1},
\]
as claimed.
\end{proof}

\begin{lemma}[{\cite[Exercise 4.7.5]{MR4655919}}]
\label{lemma-explicit-formula}
We have
\[
\sigma_i\cdot \sigma_j
=
\sigma_{i+j}
+
\sum_{j\le k<i+j} C_{i,j}^k\,\sigma_k,
\]
where
\[
C_{i,j}^k=
\sum_{1\le p_1<\cdots<p_r\le i}
\prod_{s=1}^r \bigl(t_{p_s}-t_{p_s+j+1-s}\bigr),
\]
and $r=i+j-k$.
\end{lemma}

\begin{proof}
We argue by induction on $i$, using the recurrence from Lemma~\ref{lemma-recurrence}:
\[
C_{i+1,j}^{k}=(t_{i+1}-t_{k+1})\,C_{i,j}^{k}+C_{i,j}^{k-1}.
\]
For the base case $i=1$, Lemma~\ref{lemma-pieri} gives
\[
\sigma_1\cdot \sigma_j
=
\sigma_{j+1}+(t_1-t_{j+1})\sigma_j.
\]
Hence
\[
C_{1,j}^{j+1}=1,\qquad C_{1,j}^{j}=t_1-t_{j+1},
\]
and all other coefficients vanish. This agrees with the stated formula since
\[
C_{1,j}^{j}
= t_1-t_{j+1}
=
\sum_{1\le p_1\le 1}(t_{p_1}-t_{p_1+j}).
\]
Now assume the formula holds for $C_{i,j}^k$ and $C_{i,j}^{k-1}$, and set $r=i+j-k$. By the induction hypothesis, we have
\[
C_{i,j}^k=
\sum_{1\le p_1<\cdots<p_r\le i}
\prod_{s=1}^r \bigl(t_{p_s}-t_{p_s+j+1-s}\bigr),
\]
and
\[
C_{i,j}^{k-1}
=
\sum_{1\le p_1<\cdots<p_{r+1}\le i}
\prod_{s=1}^{r+1}\bigl(t_{p_s}-t_{p_s+j+1-s}\bigr).
\]
Substituting into the recurrence, we obtain
\[
\begin{aligned}
C_{i+1,j}^{k}
&=(t_{i+1}-t_{k+1})
\sum_{1\le p_1<\cdots<p_r\le i}
\prod_{s=1}^r \bigl(t_{p_s}-t_{p_s+j+1-s}\bigr) \\
&\quad+
\sum_{1\le p_1<\cdots<p_{r+1}\le i}
\prod_{s=1}^{r+1}\bigl(t_{p_s}-t_{p_s+j+1-s}\bigr).
\end{aligned}
\]
We interpret the first term by adjoining $p_{r+1}=i+1$. Observe that
\[
t_{p_{r+1}} - t_{p_{r+1}+j+1-(r+1)}
=
t_{i+1}-t_{(i+1)+j+1-(r+1)}
=
t_{i+1}-t_{k+1}.
\]
Hence the first sum may be rewritten as
\[
\sum_{\substack{1\le p_1<\cdots<p_{r+1} = i+1}}
\prod_{s=1}^{r+1} \bigl(t_{p_s}-t_{p_s+j+1-s}\bigr).
\]
The second sum corresponds to those tuples with $p_{r+1}\le i$, i.e.,
\[
\sum_{\substack{1\le p_1<\cdots<p_{r+1} < i+1}}
\prod_{s=1}^{r+1}\bigl(t_{p_s}-t_{p_s+j+1-s}\bigr).
\]
Combining these two disjoint cases, we obtain
\[
C_{i+1,j}^{k}
=
\sum_{1\le p_1<\cdots<p_{r+1}\le i+1}
\prod_{s=1}^{r+1} \bigl(t_{p_s}-t_{p_s+j+1-s}\bigr),
\]
which is the desired formula for $C_{i+1,j}^k$. This completes the induction.
\end{proof}

Fix $P=\{p_1,\dots,p_r\}\subseteq [i]\coloneqq \{1,2,\dots,i\}$ with $p_1<\cdots<p_r$. 
Let $\mathcal{G}^i_j(P)$ denote the set of all $r$-matchings of $\mathcal{G}^i_j$ of the form $M = \{a_{p_1}b_{q_1}, \dots, a_{p_r}b_{q_r}\}$ where $(q_1, \dots, q_r)$ is a tuple of pairwise distinct elements of $[j]$. We now relate the explicit formula of Lemma~\ref{lemma-explicit-formula} to our bipartite matching model.

\begin{lemma}
\label{lemma-rule}
We have
\[
\sum_{M \in \mathcal{G}^i_j(P)} \operatorname{wt}(M) 
= 
\prod_{s=1}^r \bigl(t_{p_s}-t_{p_s+j+1-s}\bigr).
\]
\end{lemma}

\begin{proof}
Let $M = \{a_{p_1}b_{q_1}, \dots , a_{p_r}b_{q_r}\}$. Since $p_1<\cdots<p_r$, by~\eqref{eq-cross} we have
\[
X(a_{p_s}b_{q_s}) = \#\{ h < s \mid q_h < q_s \}.
\]
Recall from~\eqref{eq-wt-edge} that
\[
\operatorname{wt}(a_{p_s}b_{q_s}) = p_s + q_s - X(a_{p_s}b_{q_s}) - 1.
\]
For fixed $s$, the quantity
\[
q_s - X(a_{p_s}b_{q_s})
\]
is the rank of $q_s$ among the elements of $[j]\setminus\{q_1,\dots,q_{s-1}\}$ when listed in increasing order. Therefore it ranges over
\[
\{1,2,\dots,j+1-s\}.
\]
It follows that $\operatorname{wt}(a_{p_s}b_{q_s})$ ranges over
\[
\{p_s, p_s+1, \dots, p_s + j - s\}.
\]
Moreover, for any fixed choice of $(q_1,\dots,q_{s-1})$, there is a unique choice of $q_s$ realizing each possible value of $\operatorname{wt}(a_{p_s}b_{q_s})$ in the given interval. Thus the quantity $\operatorname{wt}(a_{p_s}b_{q_s})$ varies independently over these ranges. Summing over all matchings with fixed $P$, we obtain
\[
\begin{aligned}
\sum_{M \in \mathcal{G}^i_j(P)} \operatorname{wt}(M) 
&= \sum_{(q_1,\dots,q_r)} 
\prod_{s=1}^r \beta_{\operatorname{wt}(a_{p_s}b_{q_s})} \\
&= \prod_{s=1}^{r} \left( \sum_{m=p_s}^{p_s+j-s} \beta_m \right),
\end{aligned}
\]
where the sum runs over all tuples $(q_1,\dots,q_r)$ of distinct elements of $[j]$. Finally, using $\beta_m = t_m - t_{m+1}$, we have
\[
\sum_{m=p_s}^{p_s+j-s} \beta_m = t_{p_s} - t_{p_s+j+1-s}.
\]
This completes the proof.
\end{proof}

\begin{proof}[Proof of Theorem~\ref{thm:main}]
By Lemma~\ref{lemma-rule}, for each fixed subset $P=\{p_1,\dots,p_r\}$ we have
\[
\sum_{M \in \mathcal{G}^i_j(P)} \operatorname{wt}(M) 
= 
\prod_{s=1}^r \bigl(t_{p_s}-t_{p_s+j+1-s}\bigr).
\]
Summing over all choices of $P$ (equivalently, over all $r$-matchings of $\mathcal{G}^i_j$), we obtain
\[
\sum_{M} \operatorname{wt}(M)
= \sum_{1\le p_1<\cdots<p_r\le i}
\prod_{s=1}^r \bigl(t_{p_s}-t_{p_s+j+1-s}\bigr).
\]
This coincides with the formula for $C_{i,j}^k$ in Lemma~\ref{lemma-explicit-formula}, completing the proof.
\end{proof}

\section{Applications and remarks}

We may identify projective space with the Grassmannian $\mathbb{P}^n \cong \mathrm{Gr}(1,\mathbb{C}^{n+1})$, and thus view our results in this special case of equivariant Schubert calculus on Grassmannians. We first develop several structural consequences of our combinatorial rule, including the saturated Newton polytope property, a refinement of the equivariant saturation property, and a monomial-positivity property (Section~\ref{subsec:app}).\\ 

We then compare our formulation with existing combinatorial models, such as the puzzle rule of Knutson--Tao~\cite{MR1997946} and the edge-labeled Young tableaux rule of Thomas--Yong~\cite{MR3795480}. While these models apply in the more general Grassmannian setting, in the case of projective space our formulation makes the symmetry $C_{i,j}^k = C_{j,i}^k$ transparent. We also describe correspondences between these models and our formulation (Sections~\ref{subsec:puzzle} and~\ref{subsec:edge-labeled}).

\subsection{Applications}
\label{subsec:app}

Let $\operatorname{Supp}(C_{i,j}^k)$ denote the set of exponent vectors of monomials appearing in $C_{i,j}^k$, viewed as a polynomial in $\beta_1,\dots,\beta_n$. Recall that a polynomial has \emph{saturated Newton polytope} (SNP) if every lattice point in the convex hull of its exponent vectors corresponds to a monomial~\cite{MR4021852}. Robichaux--Yadav--Yong~\cite[Conjecture~6.11]{MR4381918} conjectured that the structure constants $C_{\lambda,\mu}^{\nu}$ in the equivariant Schubert calculus of the Grassmannian have SNP. In our setting, we confirm this conjecture. In fact, our combinatorial rule allows us to describe the Newton polytope of $C_{i,j}^k$ explicitly.\\

Let $r := i+j-k$. By the explicit formula of Lemma~\ref{lemma-explicit-formula}, we have
\begin{equation}
\label{eq-nonzero}
C_{i,j}^k \neq 0
\quad \Longleftrightarrow \quad
\max\{i,j\} \le k \le \min\{i+j,n\}.
\end{equation}
In particular, in this case $0 \le r \le \min\{i,j\}$. We define the following polytope:
\[
Q_{i,j}^k
:=
\left\{
(x_1,\dots,x_n)\in \mathbb{R}_{\ge 0}^n
\ \middle|\ 
\sum_{m=1}^n x_m = r,\quad
\sum_{m=1}^{\ell} x_m \le \ell \text{ for } 1 \le \ell \le i,\quad
x_m = 0 \text{ for } m > k
\right\}.
\]

The defining system of inequalities has a totally unimodular constraint matrix with integer right-hand side; in particular, $Q_{i,j}^k$ is an integral polytope (see \cite[p.~49]{MR85148}). We will use this fact in the proof of the following theorem.

\begin{theorem}
\label{thm-supp}
Assume that $\max\{i,j\} \le k \le \min\{i+j,n\}$. Then
\[
\operatorname{Supp}(C_{i,j}^k)=Q_{i,j}^k\cap \mathbb{Z}^n,
\]
and hence \(C_{i,j}^{k}\) has SNP.
\end{theorem}

\begin{proof}
The simple root relation $\beta_m=t_m-t_{m+1}$ gives
\[
t_{p_s}-t_{p_s+j+1-s}
=
\beta_{p_s}+\beta_{p_s+1}+\cdots+\beta_{p_s+j-s}.
\]
Substituting this into Lemma~\ref{lemma-explicit-formula}, we obtain
\[
C_{i,j}^k
=
\sum_{1\le p_1<\cdots<p_r\le i}
\prod_{s=1}^r
\bigl(\beta_{p_s}+\beta_{p_s+1}+\cdots+\beta_{p_s+j-s}\bigr).
\]

We first prove that
\[
\operatorname{Supp}(C_{i,j}^k)\subseteq Q_{i,j}^k\cap \mathbb Z^n.
\]
Take a monomial $\beta_1^{x_1}\cdots \beta_n^{x_n}$ appearing in $C_{i,j}^k$. Then it arises from a choice of indices
\[
1\le p_1<\cdots<p_r\le i
\]
and
\[
q_s\in \{p_s,p_s+1,\dots,p_s+j-s\}
\qquad (1\le s\le r),
\]
such that
\[
\beta_1^{x_1}\cdots \beta_n^{x_n}
=
\beta_{q_1}\cdots \beta_{q_r}.
\]
Hence \(x_m\ge 0\) for all \(m\), and
\begin{equation}
\label{eq-snp1}
\sum_{m=1}^n x_m=r.
\end{equation}
Next, since \(p_1<\cdots<p_r\le i\), we have \(p_s\le i-r+s\). Therefore
\[
q_s\le p_s+j-s\le (i-r+s)+j-s=k.
\]
Thus, we have
\begin{equation}
\label{eq-snp2}
x_m=0 \qquad \text{for all } m>k.
\end{equation}
Finally, let \(1\le \ell\le i\). Then
\[
\sum_{m=1}^{\ell} x_m
=
\#\{\,s \mid q_s\le \ell\,\}.
\]
If \(q_s\le \ell\), then necessarily \(p_s\le q_s\le \ell\). Since the \(p_s\) are distinct, there are at most \(\ell\) indices \(s\) with \(p_s\le \ell\). Hence
\begin{equation}
\label{eq-snp3}
\sum_{m=1}^{\ell} x_m
=
\#\{\,s \mid q_s\le \ell\,\}
\le
\#\{\,s \mid p_s\le \ell\,\}
\le \ell.
\end{equation}
By \eqref{eq-snp1}, \eqref{eq-snp2}, and \eqref{eq-snp3}, we conclude that \((x_1,\dots,x_n)\in Q_{i,j}^k\), proving the first inclusion.\\

We now prove the reverse inclusion
\[
Q_{i,j}^k\cap \mathbb Z^n\subseteq \operatorname{Supp}(C_{i,j}^k).
\]
Take any lattice point
\[
x \coloneqq (x_1,\dots,x_n)\in Q_{i,j}^k\cap \mathbb Z^n.
\]
Since $\sum_{m=1}^n x_m=r$, we may form a weakly increasing sequence
\[
q_1\le q_2\le \cdots\le q_r
\]
whose multiset contains exactly $x_m$ copies of $m$ for each $m$. Define
\begin{equation}
\label{eq-snp4}
p_s:=\max\{\,s,\ q_s-j+s\,\}
\qquad (1\le s\le r).
\end{equation}

We first verify that
\begin{equation}
\label{eq-snp5}
p_s\le q_s\le p_s+j-s.
\end{equation}
The inequality $q_s\le p_s+j-s$ follows directly from the definition of $p_s$ in \eqref{eq-snp4}. It remains to show that $p_s\le q_s$. Since $s\le r\le j$, we have
\[
q_s-j+s\le q_s.
\]
Thus it suffices to prove that $s\le q_s$. Since \(q_1\le \cdots \le q_r\), the first \(s\) terms all satisfy
\(q_h\le q_s\). Hence
\[
s \le \#\{h\mid q_h\le q_s\}
= \sum_{m=1}^{q_s} x_m.
\]
If $q_s \le i$, then by the defining inequalities of $Q_{i,j}^k$ we have
\[
s \le \sum_{m=1}^{q_s} x_m \le q_s.
\]
If $q_s > i$, then
\[
s \le r \le i < q_s.
\]
Thus, in all cases we have $s \le q_s$. Therefore $p_s \le q_s$, and \eqref{eq-snp5} follows.\\

Next, we verify that
\begin{equation}
\label{eq-snp6}
1\le p_1<\cdots<p_r\le i.
\end{equation}
The inequality $1\le p_1$ is immediate. To show $p_r\le i$, note that $r\le i$ and $q_r\le k$ (since $x_m=0$ for all $m>k$). Hence
\[
q_r-j+r \le k-j+r = i,
\]
and therefore
\[
p_r=\max\{\,r,\ q_r-j+r\,\}\le i.
\]
It remains to show that $p_s<p_{s+1}$ for $1\le s<r$. Since $q_s\le q_{s+1}$, we have
\[
q_s-j+s < q_s-j+s+1 \le q_{s+1}-j+(s+1).
\]
Together with $s<s+1$, this implies
\[
p_s = \max\{\,s,\ q_s-j+s\,\}
<
\max\{\,s+1,\ q_{s+1}-j+(s+1)\}
= p_{s+1}.
\]
Thus \eqref{eq-snp6} holds. Combining \eqref{eq-snp5} and \eqref{eq-snp6}, we get a subset $P=\{p_1,\dots,p_r\}\subset [i]$ such that
\[
q_s\in \{p_s,p_s+1,\dots,p_s+j-s\}
\quad (1\le s\le r).
\]
Hence the summand
\[
\prod_{s=1}^r
\bigl(\beta_{p_s}+\beta_{p_s+1}+\cdots+\beta_{p_s+j-s}\bigr)
\]
contains the monomial
\[
\beta_{q_1}\cdots \beta_{q_r}
=
\beta_1^{x_1}\cdots \beta_n^{x_n}.
\]
Hence $x\in \operatorname{Supp}(C_{i,j}^k)$, proving the reverse inclusion. Therefore,
\[
\operatorname{Supp}(C_{i,j}^k)=Q_{i,j}^k\cap \mathbb{Z}^n.
\]
Since $Q_{i,j}^k$ is an integral polytope, $C_{i,j}^k$ has SNP.
\end{proof}

In the general setting of equivariant Schubert calculus on Grassmannians, 
Anderson--Richmond--Yong~\cite[Theorem~1.1]{MR3109734} proved the saturation property
\[
C_{\lambda,\mu}^{\nu} \neq 0 \iff C_{N\lambda,N\mu}^{N\nu} \neq 0.
\]
In our setting, working in the equivariant cohomology of $\mathbb{P}^{\infty}$, 
the nonvanishing condition \eqref{eq-nonzero} implies the saturation property
\[
C_{i,j}^k \neq 0 \iff C_{Ni,Nj}^{Nk} \neq 0.
\]
We now strengthen this property by describing how the supports behave under scaling. In the equivariant cohomology of $\mathbb{P}^{\infty}$, we regard both $C_{Ni,Nj}^{Nk}$ and $(C_{i,j}^k)^N$ as polynomials in $\mathbb{Z} [\beta_1,\beta_2,\ldots]$. Thus their supports may be viewed as
subsets of $\mathbb{R}^{\infty}$ via their exponent vectors (with finitely many nonzero entries). Define a linear map $\pi_N:\mathbb{R}^{\infty}\to \mathbb{R}^{\infty}$ by
\[
\pi_N(x_1,x_2,\ldots)
=
\left(
\sum_{m=1}^{N}x_m,\,
\sum_{m=N+1}^{2N}x_m,\,
\sum_{m=2N+1}^{3N}x_m,\ldots
\right).
\]

\begin{theorem}
\label{thm-pi-N}
We have
\[
\pi_N\bigl(\operatorname{Supp}(C_{Ni,Nj}^{Nk})\bigr)
=
\operatorname{Supp}\bigl((C_{i,j}^k)^N\bigr).
\]
\end{theorem}

\begin{example}
Let $i=1$, $j=2$, $k=2$. Then $r=1$, and
\[
C_{1,2}^2 = \beta_1 + \beta_2.
\]
Now take $N=2$. Then
\[
C_{2,4}^4 = \beta_1 \beta_2 + \beta_1 \beta_3 + \beta_1 \beta_4
+ \beta_2^2 + 2 \beta_2 \beta_3 + 2 \beta_2 \beta_4
+ \beta_3^2 + 2 \beta_3 \beta_4 
+ \beta_4^2,
\]
and 
\[
\left( C_{1,2}^2 \right)^2 = \beta_1^2 + 2 \beta_1 \beta_2 + \beta_2^2.
\]
We may also view $\pi_N$ as an operator on $\mathbb{Z}[\beta_1, \beta_2, \dots]$ by defining
\[
\pi_N(\beta^x) := \beta^{\pi_N(x)} \quad \text{for all } x \in \mathbb{Z}_{\ge 0}^{\infty}.
\]
Then
\[
\pi_2(\beta_1)=\pi_2(\beta_2)= \beta_1, 
\qquad
\pi_2(\beta_3)=\pi_2(\beta_4)= \beta_2.
\]
This induces a map on monomials, and hence on supports:
\[ 
\begin{tabular}{|c|c|c|c|c|c|c|c|c|c|} 
\hline 
$\beta^x$ & 
$\beta_1 \beta_2$ & 
$\beta_1 \beta_3$ & 
$\beta_1 \beta_4$ & 
$\beta_2^2$ & 
$\beta_2 \beta_3$ & 
$\beta_2 \beta_4$ & 
$\beta_3^2$ & 
$\beta_3 \beta_4$ & 
$\beta_4^2$ \\ 

\hline 
$\pi_N(\beta^x)$ & 
$\beta_1^2$ & 
$\beta_1 \beta_2$ & 
$\beta_1 \beta_2$ & 
$\beta_1^2$ & 
$\beta_1 \beta_2$ & 
$\beta_1 \beta_2$ & 
$\beta_2^2$ & 
$\beta_2^2$ & 
$\beta_2^2$ \\ 
\hline 
\end{tabular} 
\]
Therefore,
\[
\pi_2\bigl(\operatorname{Supp}(C_{2,4}^{4})\bigr)
=
\operatorname{Supp}\bigl((C_{1,2}^2)^2\bigr).
\]
\end{example}

To prove Theorem~\ref{thm-pi-N}, we require the following lemma. Throughout, for a polytope $Q \subset \mathbb{R}^\infty$ and an integer $N \ge 1$, we denote by $NQ$ the scalar dilation of $Q$:
\[
NQ := \{ Nx \mid x \in Q \}.
\]

\begin{lemma}
\label{lemma-pi-N}
We have
\[
\pi_N\bigl(Q_{Ni,Nj}^{Nk}\bigr)=NQ_{i,j}^k.
\]
\end{lemma}

\begin{proof}
We compare the following two polytopes:
\[
Q_{Ni,Nj}^{Nk}
=
\left\{
x\in \mathbb{R}_{\ge 0}^{\infty}
\;\middle|\;
\sum_{m\ge 1}x_m=Nr,\ 
\sum_{m=1}^{\ell}x_m\le \ell \text{ for } 1\le \ell\le Ni,\ 
x_m=0 \text{ for } m>Nk
\right\},
\]
\[
NQ_{i,j}^k
=
\left\{
y\in \mathbb{R}_{\ge 0}^{\infty}
\;\middle|\;
\sum_{g\ge 1}y_g=Nr,\ 
\sum_{g=1}^{\ell}y_g\le N\ell \text{ for } 1\le \ell\le i,\ 
y_g=0 \text{ for } g>k
\right\}.
\]
We first prove that
\[
\pi_N\bigl(Q_{Ni,Nj}^{Nk}\bigr)\subseteq NQ_{i,j}^k.
\]
Let $x\in Q_{Ni,Nj}^{Nk}$ and set $y=\pi_N(x)$. Then $y_g\ge 0$ for all $g$, and
\[
\sum_{g\ge 1} y_g
=
\sum_{m\ge 1} x_m
=
Nr.
\]
Moreover, if $g>k$, then
\[
y_g=\sum_{m=(g-1)N+1}^{gN} x_m =0,
\]
because $(g-1)N+1>kN=Nk$, so every index in the sum is greater than \(Nk\). Finally, for
\(1\le \ell\le i\), we have
\[
\sum_{g=1}^{\ell} y_g
=
\sum_{m=1}^{N\ell} x_m
\le N\ell,
\]
since \(N\ell\le Ni\). Thus \(y\in NQ_{i,j}^k\), proving the first inclusion.\\

For the reverse inclusion, let \(y=(y_1,y_2,\dots)\in NQ_{i,j}^k\). Define
\(x\in \mathbb{R}_{\ge 0}^{\infty}\) by
\[
x_{gN}:=y_g \quad (1\le g\le k),
\qquad
x_m:=0 \quad \text{for all other } m.
\]
Then \(\pi_N(x)=y\), and clearly \(x_m=0\) for all \(m>Nk\). Moreover,
\[
\sum_{m\ge 1}x_m=\sum_{g\ge 1}y_g=Nr.
\]
It remains to verify the prefix inequalities for \(x\). Let \(1\le s\le Ni\), and write
\[
(g-1)N < s \le gN
\]
for some integer \(g\) with \(1\le g\le i\). If \(s<gN\), then
\[
\sum_{m=1}^s x_m
=
\sum_{h=1}^{g-1} y_h
\le N(g-1)
\le s.
\]
If \(s=gN\), then
\[
\sum_{m=1}^{gN} x_m
=
\sum_{h=1}^{g} y_h
\le Ng
=
s.
\]
Thus \(x\in Q_{Ni,Nj}^{Nk}\). Therefore
\[
NQ_{i,j}^k\subseteq \pi_N\bigl(Q_{Ni,Nj}^{Nk}\bigr).
\]
This completes the proof.
\end{proof}

M.~Larson (private communication) asked whether \(Q_{i,j}^k\) is a generalized permutohedron in the sense of Postnikov~\cite[Definition~6.1]{MR2487491}. This is indeed the case, as can be verified directly from the defining inequalities, which are given by prefix sum constraints and hence correspond to a submodular function \(z_I\) in Postnikov’s description.\\

Since \(Q_{i,j}^k\) is an integral generalized permutohedron, it has the integer decomposition property; see, e.g., Schrijver~\cite[Section~46.6]{MR1956926}. That is, for every integer \(N \ge 1\), we have
\begin{equation}
\label{eq-idp}
NQ_{i,j}^k \cap \mathbb{Z}^\infty
=
\bigl(Q_{i,j}^k \cap \mathbb{Z}^\infty\bigr)^{\oplus N}.
\end{equation}
Here \(S^{\oplus N}\) denotes the $N$-fold Minkowski sum of a subset \(S \subset \mathbb{R}^\infty\).

\begin{proof}[Proof of Theorem~\ref{thm-pi-N}]
By Theorem~\ref{thm-supp}, we have
\[
\operatorname{Supp}(C_{Ni,Nj}^{Nk})
=
Q_{Ni,Nj}^{Nk}\cap \mathbb{Z}^{\infty},
\qquad
\operatorname{Supp}(C_{i,j}^k)
=
Q_{i,j}^k\cap \mathbb{Z}^{\infty}.
\]
We first prove that
\[
\pi_N\bigl(\operatorname{Supp}(C_{Ni,Nj}^{Nk})\bigr)
=
NQ_{i,j}^k \cap \mathbb{Z}^{\infty}.
\]
Indeed, if
\[
x \in \operatorname{Supp}(C_{Ni,Nj}^{Nk}) = Q_{Ni,Nj}^{Nk}\cap \mathbb{Z}^{\infty},
\]
then by Lemma~\ref{lemma-pi-N},
\[
\pi_N(x)\in \pi_N\bigl(Q_{Ni,Nj}^{Nk}\bigr)=NQ_{i,j}^k.
\]
Since \(x\in \mathbb{Z}^{\infty}\) and \(\pi_N\) has integer coefficients, we also have
\[
\pi_N(x)\in \mathbb{Z}^{\infty}.
\]
Thus
\[
\pi_N\bigl(\operatorname{Supp}(C_{Ni,Nj}^{Nk})\bigr)
\subseteq
NQ_{i,j}^k \cap \mathbb{Z}^{\infty}.
\]
Conversely, let
\[
y\in NQ_{i,j}^k \cap \mathbb{Z}^{\infty}.
\]
In the proof of Lemma~\ref{lemma-pi-N}, we showed that there exists
\[
x\in Q_{Ni,Nj}^{Nk}
\]
such that \(\pi_N(x)=y\), namely by setting
\[
x_{gN}:=y_g \quad (1\le g\le k),
\qquad
x_m:=0 \quad \text{for all other } m.
\]
Since \(y\in \mathbb{Z}^{\infty}\), this \(x\) also lies in \(\mathbb{Z}^{\infty}\). Hence
\[
x\in Q_{Ni,Nj}^{Nk}\cap \mathbb{Z}^{\infty}
=
\operatorname{Supp}(C_{Ni,Nj}^{Nk}),
\]
and therefore
\[
y=\pi_N(x)\in \pi_N\bigl(\operatorname{Supp}(C_{Ni,Nj}^{Nk})\bigr).
\]
This proves
\[
\pi_N\bigl(\operatorname{Supp}(C_{Ni,Nj}^{Nk})\bigr)
=
NQ_{i,j}^k \cap \mathbb{Z}^{\infty}.
\]
On the other hand, by Theorem~\ref{thm-supp} and \eqref{eq-idp},
\[
\operatorname{Supp}(C_{i,j}^k)
=
Q_{i,j}^k\cap \mathbb{Z}^{\infty},
\]
so
\[
\operatorname{Supp}\bigl((C_{i,j}^k)^N\bigr)
=
\bigl(Q_{i,j}^k\cap \mathbb{Z}^{\infty}\bigr)^{\oplus N}
=
NQ_{i,j}^k \cap \mathbb{Z}^{\infty}.
\]
Therefore,
\[
\pi_N\bigl(\operatorname{Supp}(C_{Ni,Nj}^{Nk})\bigr)
=
\operatorname{Supp}\bigl((C_{i,j}^k)^N\bigr). \qedhere
\]
\end{proof}

Finally, we give a corollary of Theorem~\ref{thm:main} and an open problem.

\begin{corollary}
\label{cor:term-count}
The polynomial $C_{i,j}^k \in \mathbb{Z}_{\ge 0}[\beta_1, \ldots, \beta_n]$ contains exactly
\[
r! \binom{i}{r} \binom{j}{r}
\]
monomial terms, counted with multiplicity.
\end{corollary}

\begin{proof}
By Theorem~\ref{thm:main}, the monomials of $C_{i,j}^k$ are in bijection (with multiplicity) with the $r$-matchings of $\mathcal{G}^i_j$. Thus the claim reduces to counting such matchings. To form an $r$-matching, one first chooses a subset $X \subseteq A$ of size $r$ and a subset $Y \subseteq B$ of size $r$, which can be done in $\binom{i}{r}\binom{j}{r}$ ways. One then matches the vertices in $X$ with those in $Y$, and there are $r!$ such bijections.
\end{proof}

Motivated by Corollary~\ref{cor:term-count}, we say that a combinatorial rule for the general constant $C_{\lambda,\mu}^{\nu}$ in the equivariant Schubert calculus of the Grassmannian is \emph{monomial positive} if each weight assigned to a combinatorial object is a monomial in the variables $\beta_1,\dots,\beta_n$. In particular, monomial positivity directly implies Graham positivity. This leads to the following open problem.

\begin{problem}
Give a monomial-positive combinatorial rule for the constants $C_{\lambda,\mu}^{\nu}$ in the equivariant Schubert calculus of the Grassmannian.
\end{problem}

\subsection{Knutson--Tao puzzles}
\label{subsec:puzzle}

We refer to Knutson--Tao~\cite[Section~1.2]{MR1997946} for the puzzle rule. For convenience, we adopt their notation and recall their result. Let $\mathcal{P}_{i,j}^k$ denote the set of puzzles with boundary
\[
0^{n-i}\,1\,0^{i}, \quad 0^{n-j}\,1\,0^{j}, \quad 0^{n-k}\,1\,0^{k}
\]
along $AB$, $BC$, and $AC$, respectively. The left picture in Figure~\ref{fig:puzzle} shows an example of a puzzle in $\mathcal{P}_{2,3}^3$.

\begin{theorem}[{\cite[Theorem~1.2]{MR1997946}}]
We have
\[
C_{i,j}^k = \sum_{Z \in \mathcal{P}_{i,j}^k} \operatorname{wt}(Z).
\]
\end{theorem}

As illustrated in Figure~\ref{fig:puzzle}, the equivariant puzzle pieces are colored blue. The weight of the equivariant piece labeled $1$ in the right picture is $t_1 - t_3$. The weight of a puzzle is defined as the product of the weights of its equivariant pieces. Thus, the puzzle on the left has weight $(t_1 - t_4)(t_2 - t_4)$, while the one on the right has weight $(t_1 - t_3)(t_3 - t_4)$. This model does not exhibit manifest symmetry. Indeed, $\mathcal{P}_{2,3}^3$ contains only one puzzle, whereas $\mathcal{P}_{3,2}^3$ contains three puzzles. The equality holds since
\[
(t_1 - t_4)(t_2 - t_4)
=
(t_1 - t_3)(t_2 - t_3)
+
(t_1 - t_3)(t_3 - t_4)
+
(t_2 - t_4)(t_3 - t_4).
\]

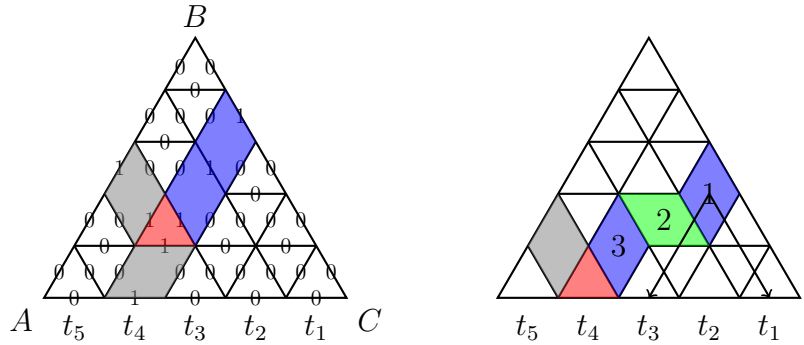
\begin{figure}[h]
\centering
  \begin{tikzpicture}[scale=0.4]
    \coordinate (A) at (0,0);
    \coordinate (C) at (10,0);
    \coordinate (B) at (5,8.6); 

    \coordinate (A') at (15,0);
    \coordinate (B') at (25,0);
    \coordinate (C') at (20,8.6);

    \node[below left] at (A) {$A$};
    \node[below right] at (C) {$C$};
    \node[above] at (B) {$B$};
 
    \draw[thick] (A) -- (B) -- (C) -- cycle;
    \draw[thick] (A') -- (B') -- (C') -- cycle;

    \node[scale=0.7] at (1,0) {$0$};
    \node[scale=0.7] at (3,0) {$1$};
    \node[scale=0.7] at (5,0) {$0$};
    \node[scale=0.7] at (7,0) {$0$};
    \node[scale=0.7] at (9,0) {$0$};

    \node[scale=1] at (1,-1) {$t_5$};
    \node[scale=1] at (3,-1) {$t_4$};
    \node[scale=1] at (5,-1) {$t_3$};
    \node[scale=1] at (7,-1) {$t_2$};
    \node[scale=1] at (9,-1) {$t_1$}; 

    \node[scale=0.7] at (0.5,0.86) {$0$};
    \node[scale=0.7] at (1.5,2.58) {$0$};
    \node[scale=0.7] at (2.5,4.30) {$1$};
    \node[scale=0.7] at (3.5,6.02) {$0$};
    \node[scale=0.7] at (4.5,7.65) {$0$};

    \node[scale=0.7] at (9.5,0.86) {$0$};
    \node[scale=0.7] at (8.5,2.58) {$0$};
    \node[scale=0.7] at (7.5,4.30) {$0$};
    \node[scale=0.7] at (6.5,6.02) {$1$};
    \node[scale=0.7] at (5.5,7.65) {$0$};

    \draw[thick] (1,1.72) -- (2,0);
    \draw[thick] (2,3.44) -- (3,1.72);
    \draw[thick] (3,5.16) -- (6,0);
    \draw[thick] (4,6.88) -- (8,0);

    \draw[thick] (6,6.88) -- (2,0);
    \draw[thick] (7,5.16) -- (4,0);
    \draw[thick] (8,3.44) -- (6,0);
    \draw[thick] (9,1.72) -- (8,0);

    \draw[thick] (1,1.72) -- (9,1.72);
    \draw[thick] (6,3.44) -- (8,3.44);
    \draw[thick] (3,5.16) -- (5,5.16);
    \draw[thick] (4,6.88) -- (6,6.88);

    \node[scale=0.7] at (4,1.72) {$1$};
    \node[scale=0.7] at (3.5,2.58) {$1$};
    \node[scale=0.7] at (4.5,2.58) {$1$};
    \node[scale=0.7] at (5.5,4.30) {$1$};

    \node[scale=0.7] at (1.5,0.86) {$0$};
    \node[scale=0.7] at (2.5,0.86) {$0$};
    \node[scale=0.7] at (4.5,0.86) {$0$};
    \node[scale=0.7] at (5.5,0.86) {$0$};
    \node[scale=0.7] at (6.5,0.86) {$0$};
    \node[scale=0.7] at (7.5,0.86) {$0$};
    \node[scale=0.7] at (8.5,0.86) {$0$};

    \node[scale=0.7] at (2.5,2.58) {$0$};
    \node[scale=0.7] at (5.5,2.58) {$0$};
    \node[scale=0.7] at (6.5,2.58) {$0$};
    \node[scale=0.7] at (7.5,2.58) {$0$};
    \node[scale=0.7] at (3.5,4.30) {$0$};
    \node[scale=0.7] at (4.5,4.30) {$0$};
    \node[scale=0.7] at (6.5,4.30) {$0$};
    \node[scale=0.7] at (4.5,6.02) {$0$};
    \node[scale=0.7] at (5.5,6.02) {$0$};

    \node[scale=0.7] at (5,6.88) {$0$};
    \node[scale=0.7] at (4,5.16) {$0$};
    \node[scale=0.7] at (7,3.44) {$0$};
    \node[scale=0.7] at (2,1.72) {$0$};
    \node[scale=0.7] at (6,1.72) {$0$};
    \node[scale=0.7] at (8,1.72) {$0$};

    \fill[red, opacity=0.5] (3,1.72) -- (5,1.72) -- (4,3.44) -- cycle;

    \fill[gray, opacity=0.5] (2,3.44) -- (3,1.72) -- (4,3.44) -- (3,5.16) -- cycle;

    \fill[gray, opacity=0.5] (2,0) -- (4,0) -- (5,1.72) -- (3,1.72) -- cycle;

    \fill[blue, opacity=0.5] (4,3.44) -- (5,1.72) -- (6,3.44) -- (5,5.16) -- cycle;

    \fill[blue, opacity=0.5] (5,5.16) -- (6,3.44) -- (7,5.16) -- (6,6.88) -- cycle;
    
    \draw[thick] (20,1.72) -- (24,1.72);
    \draw[thick] (17,3.44) -- (21,3.44);
    \draw[thick] (18,5.16) -- (22,5.16);
    \draw[thick] (19,6.88) -- (21,6.88);
    \draw[thick] (17,0) -- (21,6.88);
    \draw[thick] (19,0) -- (20,1.72);
    \draw[thick] (21,3.44) -- (22,5.16);
    \draw[thick] (21,0) -- (23,3.44);
    \draw[thick] (23,0) -- (24,1.72);
    \draw[thick] (23,0) -- (19,6.88);
    \draw[thick] (21,0) -- (18,5.16);
    \draw[thick] (19,0) -- (17,3.44);
    \draw[thick] (17,0) -- (16,1.72);

    \fill[red, opacity=0.5] (17,0) -- (19,0) -- (18,1.72) -- cycle;
    \fill[gray, opacity=0.5] (16,1.72) -- (17,0) -- (18,1.72) -- (17,3.44) -- cycle;
    \fill[blue, opacity=0.5] (18,1.72) -- (19,0) -- (20,1.72) -- (19,3.44) -- cycle;
    \fill[blue, opacity=0.5] (21,3.44) -- (22,1.72) -- (23,3.44) -- (22,5.16) -- cycle;
    \fill[green, opacity=0.5] (19,3.44) -- (20,1.72) -- (22,1.72) -- (21,3.44) -- cycle;

    \node[scale=1] at (16,-1) {$t_5$};
    \node[scale=1] at (18,-1) {$t_4$};
    \node[scale=1] at (20,-1) {$t_3$};
    \node[scale=1] at (22,-1) {$t_2$};
    \node[scale=1] at (24,-1) {$t_1$};

    \draw[thick,->] (22,3.44) -- (24,0);
    \draw[thick,->] (22,3.44) -- (20,0);

    \node[scale=1] at (22,3.44) {$1$};
    \node[scale=1] at (20.5,2.58) {$2$};
    \node[scale=1] at (19,1.72) {$3$};
    
\end{tikzpicture}  
\caption{Puzzles in $\mathcal{P}_{2,3}^3$ and $\mathcal{P}_{3,2}^3$ with $n=4$.}
\label{fig:puzzle}
\end{figure}

We now relate the puzzle rule to our formulation. In the case of projective space, puzzles admit a particularly simple structure. From the $1$'s on $AB$ and $AC$, we trace straight gray bands that meet at a common red triangle. The band from the $1$ on $BC$ to this red triangle is colored blue if it corresponds to an equivariant puzzle piece, and green otherwise (see Figure~\ref{fig:puzzle}). This third band has length $i$ and contains $r$ blue pieces. Recording positions of blue pieces yields a subset $P = \{p_1, \dots, p_r\} \subset [i]$. For example, in the right picture of Figure~\ref{fig:puzzle}, we have $P = \{1,3\} \subset [3]$. For each subset $P \subset [i]$ with $|P|=r$, there is a unique corresponding puzzle, which we denote by $Z_P$.

\begin{proposition}
For each subset $P \subset [i]$ with $|P| = r$, we have
\[
\operatorname{wt}(Z_P)
=
\sum_{M \in \mathcal{G}^i_j(P)} \operatorname{wt}(M).
\]
\end{proposition}

\begin{proof}
By construction, the equivariant puzzle pieces in $Z_P$ occur exactly at positions $p_1,\dots,p_r$. The weight contributed by the piece at position $p_s$ is
\[
t_{p_s} - t_{p_s + j + 1 - s}.
\]
Hence,
\[
\operatorname{wt}(Z_P)
=
\prod_{s=1}^r \bigl(t_{p_s} - t_{p_s + j + 1 - s}\bigr).
\]
The claim now follows from Lemma~\ref{lemma-rule}.
\end{proof}

\subsection{Thomas--Yong edge-labeled tableaux}
\label{subsec:edge-labeled}

Following the notation of Thomas--Yong~\cite[Section~1.4]{MR3795480}, let $\mathcal{F}_{i,j}^k$ denote the set of valid (nonzero-weight) equivariant fillings of $(k/i,j)$.
\begin{theorem}[{\cite[Theorem~1.2]{MR3795480}}]
We have
\[
C_{i,j}^k = \sum_{T \in \mathcal{F}_{i,j}^k} \operatorname{wt}(T).
\]
\end{theorem}
Let $r \coloneqq i+j-k$. Each valid equivariant filling corresponds to a subset $P \subset [i]$ with $|P|=r$. For such a subset $P = \{p_1, \dots, p_r\}$ with $p_1 < \cdots < p_r$, we construct the equivariant filling $T_P \in \mathcal{F}_{i,j}^k$ by placing the edge labels $1,2,\dots,r$ below the boxes in positions $p_1,\dots,p_r$, and filling the remaining boxes in positions $[k]\setminus [i]$ with labels from $[j]\setminus [r]$. For example, in the left picture of Figure~\ref{fig:edge}, we have $T_{\{1,3\}} \in \mathcal{F}_{3,4}^5$. We now describe the weight of $T_P$ in this model using a simple observation from \cite[Section~1.4]{MR3795480}. For each edge label $s \in [r]$, we have
\[
\operatorname{wt}(s) = t_{p_s} - t_{p_s + j + 1 - s}.
\]
Thus, the weight of $T_P$ is given by
\[
\operatorname{wt}(T_P) = \prod_{s=1}^r \bigl(t_{p_s} - t_{p_s + j + 1 - s}\bigr).
\]
For example, in the left picture of Figure~\ref{fig:edge}, the edge label $2$ contributes the weight $t_3 - t_6$. The filling $T_{\{1,3\}} \in \mathcal{F}_{3,4}^5$ has weight $(t_1 - t_5)(t_3 - t_6)$, while $T_{\{2,3\}} \in \mathcal{F}_{4,3}^5$ has weight $(t_2 - t_5)(t_3 - t_5)$. As in the puzzle rule, this model does not exhibit manifest symmetry in $i$ and $j$. However, it admits a direct correspondence with our symmetric formulation.

\begin{figure}[h]
\centering
\begin{tikzpicture}[scale=0.7, every node/.style={font=\small}]

\draw (0,0) -- (5,0);
\draw (0,1) -- (5,1);
\draw (0,0) -- (0,1);
\draw (1,0) -- (1,1);
\draw (2,0) -- (2,1);
\draw (3,0) -- (3,1);
\draw (4,0) -- (4,1);
\draw (5,0) -- (5,1);

\draw (7,0) -- (12,0);
\draw (7,1) -- (12,1);
\draw (7,0) -- (7,1);
\draw (8,0) -- (8,1);
\draw (9,0) -- (9,1);
\draw (10,0) -- (10,1);
\draw (11,0) -- (11,1);
\draw (12,0) -- (12,1);

\node at (0.5,0) {1};
\node at (2.5,0) {2};
\node at (3.5,0.5) {3};
\node at (4.5,0.5) {4};

\node at (8.5,0) {1};
\node at (9.5,0) {2};
\node at (11.5,0.5) {3};

\node at (0.5,1.5) {\small $1^{\text{st}}$};
\node at (1.5,1.5) {\small $2^{\text{nd}}$};
\node at (2.5,1.5) {\small $3^{\text{rd}}$};
\node at (3.5,1.5) {\small $4^{\text{th}}$};
\node at (4.5,1.5) {\small $5^{\text{th}}$};

\node at (7.5,1.5) {\small $1^{\text{st}}$};
\node at (8.5,1.5) {\small $2^{\text{nd}}$};
\node at (9.5,1.5) {\small $3^{\text{rd}}$};
\node at (10.5,1.5) {\small $4^{\text{th}}$};
\node at (11.5,1.5) {\small $5^{\text{th}}$};

\end{tikzpicture}
\caption{Edge-labeled fillings in $\mathcal{F}_{3,4}^5$ and $\mathcal{F}_{4,3}^5$.}
\label{fig:edge}
\end{figure}
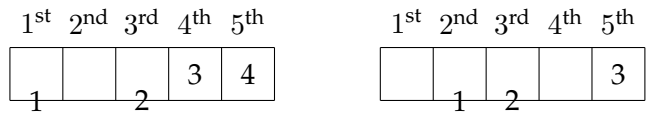

\begin{proposition}
For each subset $P \subset [i]$ with $|P| = r$, we have
\[
\operatorname{wt}(T_P)
=
\sum_{M \in \mathcal{G}^i_j(P)} \operatorname{wt}(M).
\]
\end{proposition}

\begin{proof}
By the formula for $\operatorname{wt}(T_P)$, the weight is given by the same product as in Lemma~\ref{lemma-rule}. Therefore, the claim follows.
\end{proof}

\section*{Acknowledgements}
I am grateful to Alexander Yong for suggesting this problem and for many helpful discussions and guidance throughout this work. I also thank Dave Anderson for valuable discussions.

\bibliographystyle{abbrv}
\bibliography{bibliography.bib}

@book {MR4655919,
    AUTHOR = {Anderson, David and Fulton, William},
     TITLE = {Equivariant cohomology in algebraic geometry},
    SERIES = {Cambridge Studies in Advanced Mathematics},
    VOLUME = {210},
 PUBLISHER = {Cambridge University Press, Cambridge},
      YEAR = {2024},
     PAGES = {xv+446},
      ISBN = {978-1-00-934998-7},
   MRCLASS = {14L30 (05E14 14-02 14F43 14Mxx 20G05 55N91)},
  MRNUMBER = {4655919},
MRREVIEWER = {Michael\ Orin\ Joyce},
}

@article {MR3795480,
    AUTHOR = {Thomas, Hugh and Yong, Alexander},
     TITLE = {Equivariant {S}chubert calculus and jeu de taquin},
   JOURNAL = {Ann. Inst. Fourier (Grenoble)},
  FJOURNAL = {Universit\'{e} de Grenoble. Annales de l'Institut Fourier},
    VOLUME = {68},
      YEAR = {2018},
    NUMBER = {1},
     PAGES = {275--318},
      ISSN = {0373-0956,1777-5310},
   MRCLASS = {05E10 (05E05 14N15)},
  MRNUMBER = {3795480},
MRREVIEWER = {Evgeny\ Smirnov},
       DOI = {10.5802/aif.3161},
       URL = {https://doi.org/10.5802/aif.3161},
}

@article {MR1853356,
    AUTHOR = {Graham, William},
     TITLE = {Positivity in equivariant {S}chubert calculus},
   JOURNAL = {Duke Math. J.},
  FJOURNAL = {Duke Mathematical Journal},
    VOLUME = {109},
      YEAR = {2001},
    NUMBER = {3},
     PAGES = {599--614},
      ISSN = {0012-7094,1547-7398},
   MRCLASS = {14M17 (14C17 14F43 14M15 17B37)},
  MRNUMBER = {1853356},
MRREVIEWER = {E.\ Aky\i ld\i z},
       DOI = {10.1215/S0012-7094-01-10935-6},
       URL = {https://doi.org/10.1215/S0012-7094-01-10935-6},
}

@article {MR1997946,
    AUTHOR = {Knutson, Allen and Tao, Terence},
     TITLE = {Puzzles and (equivariant) cohomology of {G}rassmannians},
   JOURNAL = {Duke Math. J.},
  FJOURNAL = {Duke Mathematical Journal},
    VOLUME = {119},
      YEAR = {2003},
    NUMBER = {2},
     PAGES = {221--260},
      ISSN = {0012-7094,1547-7398},
   MRCLASS = {14N15 (05E05 05E10 57R91 57S25)},
  MRNUMBER = {1997946},
       DOI = {10.1215/S0012-7094-03-11922-5},
       URL = {https://doi.org/10.1215/S0012-7094-03-11922-5},
}

@article {MR3109734,
    AUTHOR = {Anderson, David and Richmond, Edward and Yong, Alexander},
     TITLE = {Eigenvalues of {H}ermitian matrices and equivariant cohomology
              of {G}rassmannians},
   JOURNAL = {Compos. Math.},
  FJOURNAL = {Compositio Mathematica},
    VOLUME = {149},
      YEAR = {2013},
    NUMBER = {9},
     PAGES = {1569--1582},
      ISSN = {0010-437X,1570-5846},
   MRCLASS = {14M15 (05E15 14F43 15A18 57R91)},
  MRNUMBER = {3109734},
MRREVIEWER = {Huajun\ Huang},
       DOI = {10.1112/S0010437X13007343},
       URL = {https://doi.org/10.1112/S0010437X13007343},
}

@incollection {MR4381918,
    AUTHOR = {Robichaux, Colleen and Yadav, Harshit and Yong, Alexander},
     TITLE = {Equivariant cohomology, {S}chubert calculus, and edge labeled
              tableaux},
 BOOKTITLE = {Facets of algebraic geometry. {V}ol. {II}},
    SERIES = {London Math. Soc. Lecture Note Ser.},
    VOLUME = {473},
     PAGES = {284--335},
 PUBLISHER = {Cambridge Univ. Press, Cambridge},
      YEAR = {2022},
      ISBN = {978-1-108-79251-6; 978-1-108-87006-1},
   MRCLASS = {05E14 (14M15 14N15)},
  MRNUMBER = {4381918},
}

@incollection {MR85148,
    AUTHOR = {Hoffman, A. J. and Kruskal, J. B.},
     TITLE = {Integral boundary points of convex polyhedra},
 BOOKTITLE = {Linear inequalities and related systems},
    SERIES = {Ann. of Math. Stud., no. 38},
     PAGES = {223--246},
 PUBLISHER = {Princeton Univ. Press, Princeton, NJ},
      YEAR = {1956},
   MRCLASS = {90.0X},
  MRNUMBER = {85148},
MRREVIEWER = {D.\ Gale},
}

@article {MR4021852,
    AUTHOR = {Monical, Cara and Tokcan, Neriman and Yong, Alexander},
     TITLE = {Newton polytopes in algebraic combinatorics},
   JOURNAL = {Selecta Math. (N.S.)},
  FJOURNAL = {Selecta Mathematica. New Series},
    VOLUME = {25},
      YEAR = {2019},
    NUMBER = {5},
     PAGES = {Paper No. 66, 37},
      ISSN = {1022-1824,1420-9020},
   MRCLASS = {05E05 (05E10)},
  MRNUMBER = {4021852},
MRREVIEWER = {Allan\ Berele},
       DOI = {10.1007/s00029-019-0513-8},
       URL = {https://doi.org/10.1007/s00029-019-0513-8},
}

@article {MR2487491,
    AUTHOR = {Postnikov, Alexander},
     TITLE = {Permutohedra, associahedra, and beyond},
   JOURNAL = {Int. Math. Res. Not. IMRN},
  FJOURNAL = {International Mathematics Research Notices. IMRN},
      YEAR = {2009},
    NUMBER = {6},
     PAGES = {1026--1106},
      ISSN = {1073-7928,1687-0247},
   MRCLASS = {05E30},
  MRNUMBER = {2487491},
       DOI = {10.1093/imrn/rnn153},
       URL = {https://doi.org/10.1093/imrn/rnn153},
}

@book {MR1956926,
    AUTHOR = {Schrijver, Alexander},
     TITLE = {Combinatorial optimization. {P}olyhedra and efficiency. {V}ol.
              {C}},
    SERIES = {Algorithms and Combinatorics},
    VOLUME = {24},
      NOTE = {Disjoint paths, hypergraphs,
              Chapters 70--83},
 PUBLISHER = {Springer-Verlag, Berlin},
      YEAR = {2003},
     PAGES = {i--xxxiv and 1219--1881},
      ISBN = {3-540-44389-4},
   MRCLASS = {90-02 (05-02 52B55 68Q25 68R10 90C27 90C35 90C57)},
  MRNUMBER = {1956926},
MRREVIEWER = {Alexander\ I.\ Barvinok},
}

\end{document}